\documentclass[10pt]{amsart}
\usepackage{color}
\usepackage{amssymb}
\usepackage{epsfig}
\usepackage{url}
\usepackage{setspace}
\usepackage{pdflscape}
\usepackage{sagetex}
\theoremstyle{plain}

\newtheorem{thm}{Theorem}[section]
\newtheorem*{prob*}{Problem}
\newtheorem{cor}[thm]{Corollary}
\newtheorem{lem}[thm]{Lemma}

\newtheorem{rem}[thm]{Remark}

\theoremstyle{definition}
\def\cal{\mathcal}
\def\bbb{\mathbb}

\renewcommand{\phi}{\varphi}
\newcommand{\R}{\bbb{R}}

\newcommand{\Z}{\bbb{Z}}
\newcommand{\Q}{\bbb{Q}}

\begin{document}

\title[On a problem of Peth\H{o}]{On a problem of Peth\H{o}}
\author{Sz. Tengely}
\address{Institute of Mathematics\newline
 \indent University of Debrecen\newline
 \indent P.O.Box 12\newline
 \indent 4010 Debrecen\newline
 \indent Hungary}
\email{tengely@science.unideb.hu}
\author{M. Ulas}
\address{Jagiellonian University\newline
\indent Faculty of Mathematics and Computer Science \newline
\indent Institute of Mathematics\newline
 \indent \L{}ojasiewicza 6\newline
 \indent 30-348 Krak\'ow\newline
 \indent Poland}
\email{Maciej.Ulas@im.uj.edu.pl}

\thanks{Research supported in part by the OTKA grants K115479 and NK104208.}

\begin{abstract}
In this paper we deal with a problem of Peth\H{o} related to existence of quartic algebraic integer $\alpha$ for which
$$
\beta=\frac{4\alpha^4}{\alpha^4-1}-\frac{\alpha}{\alpha-1}
$$
is a quadratic algebraic number. By studying rational solutions of certain Diophantine system we prove that there are infinitely many $\alpha$'s such that the corresponding $\beta$ is quadratic. Moreover, we present a description of quartic numbers $\alpha$ such that the corresponding $\beta$ is a quadratic real number.
\end{abstract}

\bibliographystyle{plain}
\maketitle

\section{introduction}
Buchmann and Peth\H{o} \cite{BuPe} found an interesting unit in the number field $K=\mathbb{Q}(\alpha)$ with $\alpha^7-3=0$ it is as follows
$$
10+9\alpha+8\alpha^2+7\alpha^3+6\alpha^4+5\alpha^5+4\alpha^6.
$$
That is the coordinates $(x_0,\ldots,x_6)\in\mathbb{Z}^7$ of a solution of the norm form equation $N_{K/\mathbb{Q}}(x_0+x_1\alpha+\ldots+x_{6}\alpha^6)=1$ form an arithmetic progression.
In \cite{BePe0} B\'erczes and Peth\H{o} considered norm form equations
\begin{equation}\label{BP}
N_{K/\mathbb{Q}}(x_0+x_1\alpha+\ldots+x_{n-1}\alpha^{n-1})=m\quad \mbox{ in } x_0,x_1,\ldots,x_{n-1}\in\mathbb{Z}
\end{equation}
where $K=\mathbb{Q}(\alpha)$ is an algebraic number field of degree $n$ and $m$
is a given integer such that $x_0,x_1,\ldots,x_{n-1}$ are consecutive terms in an arithmetic progression. They proved that \eqref{BP} has only finitely many solutions if neither of the following two cases hold:
\begin{itemize}
\item $\alpha$ has minimal polynomial of the form
$$x^n-bx^{n-1}-\ldots-bx+(bn+b-1)$$
with $b\in\mathbb{Z}.$
\item $\frac{n\alpha^n}{\alpha^n-1}-\frac{\alpha}{\alpha-1}$ is a real quadratic number.
\end{itemize}
In 2006 B\'erczes, Peth\H{o} and Ziegler \cite{BPZ} studied norm form equations related to Thomas polynomials such that the solutions are coprime integers in arithmetic progression. B\'erczes and Peth\H{o} \cite{BePe1} considered \eqref{BP} with $m=1$ and $\alpha$ is a root of $x^n-T, (n\geq 3, 4\leq T\leq 100).$ They proved that the norm form equation has no solution in integers which are consecutive elements in an arithmetic progression.
In 2010 Peth\H{o} \cite{P15} collected 15 problems in number theory, Problem 6 is
based on the results given in \cite{BePe0}.
\begin{prob*}[Problem 6 in \cite{P15}]
Does there exist infinitely many quartic algebraic integers $\alpha$ such
that
$$
\frac{4\alpha^4}{\alpha^4-1}-\frac{\alpha}{\alpha-1}
$$
is a quadratic algebraic number?
\end{prob*}
The only example mentioned is $x^4+2x^3+5x^2+4x+2$ such that the corresponding element is a real quadratic number (that is a root of $x^2-4x+2$).
Moreover, B\'erczes, Peth\H{o} in \cite{BePe0} remarks that there are many solutions if we drop assumption of integrality of $\alpha$. However, it is not quite clear whether we can find infinitely many such examples and whether we can find precise description of such algebraic numbers. As we will see the problem is equivalent with the study of existence of rational zeros of family of four polynomials in six variables. With the help of Gr\"obner bases approach we reduce our problem to the study of rational zeros of only one (reducible) polynomial. A careful analysis of the corresponding variety allow us to get 2 infinite families of quartic polynomials defining quartic algebraic integers such that the algebraic number $\frac{4\alpha^4}{\alpha^4-1}-\frac{\alpha}{\alpha-1}$ is quadratic. Unfortunately, in this case we get real quadratic number only in finitely many cases. However, with different look we are able to show that the set of quartic algebraic numbers such that the algebraic number $\frac{4\alpha^4}{\alpha^4-1}-\frac{\alpha}{\alpha-1}$ is quadratic, is contained in a certain set given by (explicit) system of algebraic inequalities.

In particular the following is true:
\begin{thm}\label{6th}
	There are infinitely many quartic algebraic integers defined by $\alpha^4+a\alpha^3+b\alpha^2+c\alpha+d=0$
	for which
	$$
	\beta=\frac{4\alpha^4}{\alpha^4-1}-\frac{\alpha}{\alpha-1}
	$$
	is a quadratic algebraic number. Moreover, there are infinitely many quartic algebraic numbers $\alpha$ such that $\beta$ is real quadratic.
\end{thm}
\section{Auxiliary results}
We provide two infinite families of quartic polynomials, now we prove that among these polynomials there are infinitely many irreducible ones.
\begin{lem}\label{f1}
Let $t\in\mathbb{Z}.$ The polynomials defined by
$$
f_1(x)=x^4+2x^3+(2t^2+2)x^2+(4t^2-4t+2)x+6t^2-4t+1
$$
are irreducible if $t\notin\{0,1\}.$
\end{lem}
\begin{proof}
If there is a linear factor of $f_1,$ then there is an integral root. Hence we have that
$$
f_1(x)=(x+s_1)(x^3+s_2x^2+s_3x+s_4).
$$
By comparing coefficients one gets that $s_2=2-s_1,s_3=s_1^2 + 2t^2 - 2s_1+2.$
It remains to deal with the system of equations
\begin{eqnarray*}
-s_1s_4 + 6t^2 - 4t + 1&=&0\\
-s_1^3 - 2s_1t^2 + 2s_1^2 + 4t^2 - 2s_1 - s_4 - 4t + 2&=&0.
\end{eqnarray*}
The resultant of the two polynomials with respect to $s_4$ is quadratic in $t.$ The discriminant of this quadratic polynomial is
$$
(-8)(s_1 - 1)^2(s_1^4 - 2s_1^3 + 4s_1^2 - 2s_1 + 1).
$$
If $s_1=1,$ then we obtain that $t=0.$ In this case $f_1(x)=(x + 1)^2(x^2 + 1)$ is reducible. If $s_1\neq 0,$ then $-2(s_1^4 - 2s_1^3 + 4s_1^2 - 2s_1 + 1)=U^2.$ This equations has no rational solution since $-2(s_1^4 - 2s_1^3 + 4s_1^2 - 2s_1 + 1)<0$ for all $s_1\in\mathbb{Q}.$
If there are two quadratic factors, then
$$
f_1(x)=(x^2+s_1x+s_2)(x^2+s_3x+s_4).
$$
As in the previous case we compare coefficients to obtain a system of equations
\begin{eqnarray*}
-s_1^2s_2 - 2s_2t^2 + 2s_1s_2 + s_2^2 + 6t^2 - 2s_2 - 4t + 1&=&0\\
-s_1^3 - 2s_1t^2 + 2s_1^2 + 2s_1s_2 + 4t^2 - 2s_1 - 2s_2 - 4t + 2&=&0.
\end{eqnarray*}
The resultant of the above equations with respect to $s_2$ is
$$
(-1)(s_1^2 + 2t^2 - 2s_1 - 4t + 2)(s_1^4 + 2s_1^2t^2 - 4s_1^3 + 4s_1^2t - 4s_1t^2 + 6s_1^2 - 8s_1t - 4s_1 + 8t).
$$
If $s_1^2 + 2t^2 - 2s_1 - 4t + 2=0,$ then we have a quadratic polynomial in $t$ with discriminant $-8s_1^2 + 16s_1.$ It is non-negative if $s_1\in\{0,1,2\}.$ If $s_1\in\{0,1,2\},$ then $t=0$ or $t=1.$ Earlier we handled the case with $t=0,$ if $t=1,$ then we have $f_1(x)=(x^2 + 1)(x^2 + 2x + 3).$
If $s_1^4 + 2s_1^2t^2 - 4s_1^3 + 4s_1^2t - 4s_1t^2 + 6s_1^2 - 8s_1t - 4s_1 + 8t=0,$ then the discriminant with respect to $t$ is
$$
(-8)(s_1^2 - 2s_1 + 2)(s_1^4 - 4s_1^3 + 2s_1^2 + 4s_1 - 4).
$$
It remains to determine the rational points on the genus 2 curve
$$
C:\;(-8)(s_1^2 - 2s_1 + 2)(s_1^4 - 4s_1^3 + 2s_1^2 + 4s_1 - 4)=U^2.
$$
In order to do that let us note that there is a rational map $\phi:\;C\ni (s_{1},U)\mapsto (X,U)\in E$, where
$$
E:\;U^2=X^3+6X^2-20X+8,
$$
and
$$
\phi(s_{1},U)=(-2(s_{1}-1)^2,U).
$$
Using MAGMA we obtain that the rank of Mordell-Weil group is 0 with $\operatorname{Tors}(E(\Q))=\{\cal{O},(-2,\pm 8),(2,0)\}$. These torsion points yield affine rational points on the curve $C$ of the following form
$$
(0,\pm 4), (2,\pm 4).
$$
Thus $s_1\in\{0,2\}$ and it follows that $t=0,$ a case considered above.

\end{proof}

\begin{lem}\label{f2}
	Let $t\in\mathbb{Z}.$
	The polynomials defined by
	$$
	f_2(x)=x^4+2tx^3+(t^2+2t+2)x^2+(2t^2+2t)x+3t^2-2t+1
	$$
	are irreducible if $t\notin\{0,2\}.$
\end{lem}
\begin{proof}
The approach we apply here is similar to that used in the proof of the previous lemma, therefore here we only indicate the main steps. First we try to determine linear factors, that we write
$$
f_2(x)=(x+s_1)(x^3+s_2x^2+s_3x+s_4).
$$
We have that $s_2=2t-s_1,s_3=s_1^2 - 2s_1t + t^2+ 2t + 2$ and it remains to deal with the system of equations
\begin{eqnarray*}
-s_1s_4 + 3t^2 - 2t + 1&=&0\\
-s_1^3 + 2s_1^2t - s_1t^2 - 2s_1t + 2t^2 - 2s_1 - s_4 + 2t&=&0.
\end{eqnarray*}
The resultant of the two polynomials with respect to $s_4$ is quadratic in $t.$ The discriminant of this quadratic polynomial is
$$
(-8)(s_1^4 - 2s_1^3 + 4s_1^2 - 2s_1 + 1).
$$
This expression is negative for all rational $s_1,$ hence there exists no rational solution in $t.$

If there are two quadratic factors, then
$$
f_2(x)=(x^2+s_1x+s_2)(x^2+s_3x+s_4).
$$
As in the previous case we compare coefficients to obtain a system of equations
\begin{eqnarray*}
-s_1^2s_2 + 2s_1s_2t - s_2t^2 + s_2^2 - 2s_2t + 3t^2 - 2s_2 - 2t + 1	&=&0\\
-s_1^3 + 2s_1^2t - s_1t^2 + 2s_1s_2 - 2s_1t - 2s_2t + 2t^2 - 2s_1 + 2t	&=&0.
\end{eqnarray*}
The latter equation can be written as
$$
(-1)(-s_1 + t)(-s_1^2 + s_1t + 2s_2 - 2t - 2)=0.
$$
If $s_1=t,$ then $s_2^2 - 2s_2t + 3t^2 - 2s_2 - 2t + 1=0.$ The discriminant of this equation with respect to $s_2$ is $(-8)t(t - 2).$ Hence $t\in\{0,2\}.$ If $t=0,$ then $f_2(x)=(x^2+1)^2.$ If $t=2,$ then $f_2(x)=(x^2 + 2x + 3)^2.$ Consider the case $-s_1^2 + s_1t + 2s_2 - 2t - 2=0.$ We get that $s_2=\frac{s_1^2-s_1t+2t+2}{2}.$ Thus we obtain a polynomial equation only in $s_1$ and $t$ given by
$$
(1/4)(-s_1^4 + 4s_1^3t - 5s_1^2t^2 + 2s_1t^3 - 4s_1^2t + 8s_1t^2 - 4t^3 - 4s_1^2 + 8s_1t + 4t^2 - 16t)=0.
$$
The discriminant with respect to $s_1$ factors as follows
$$
(-1/32)t(t - 2)(t^4 - 8t^3 + 40t^2 - 32t + 16)^2.
$$
The latter expression is a square only if $t=0$ or $t=2,$ so we do not get new reducible polynomials.
\end{proof}

\section{proof of theorem \ref{6th}}
\begin{proof}
Let $f(x)=x^4+ax^3+bx^2+cx+d$ with $a,b,c,d\in\mathbb{Z}$ and $g(x)=x^2+px+q$ with $p,q\in\mathbb{Q}.$
Assume that $\alpha$ is a root of $f(x)$ and $\beta=\frac{4\alpha^4}{\alpha^4-1}-\frac{\alpha}{\alpha-1}$
is a root of $g(x).$ From $g(\beta)=0$ we get a degree 6 polynomial for which $\alpha$ is a root. Therefore
it is divisible by $f(x).$ Computing the reminder we obtain a cubic polynomial which has to be zero. In SageMath \cite{sage} we may compute it as follows
\begin{sageblock}
var('X,u,v')
P.<d,p,q,a,b,c>=PolynomialRing(QQ,6,order='lex')
Px.<x>=PolynomialRing(P)
w=4*X^4/(X^4-1)-X/(X-1)
W=(w^2+p*w+q).numerator()
Wrem=Px(W(X=x))%(x^4+a*x^3+b*x^2+c*x+d)
Wcoeff=Wrem.coefficients()
\end{sageblock}
We obtain the following coefficients
{\small\begin{eqnarray*}
e_1:&&-3 d p a^{2} + 5 d p a + 3 d p b - 6 d p -  d q a^{2} + 2 d q a + d q b - 3 d q - 9 d a^{2} + 12 d a + 9 d b - 10 d + q,\\
e_2:&&3 d p a - 5 d p + d q a - 2 d q + 9 d a - 12 d - 3 p a^{2} c + 5 p a c + 3 p b c - 6 p c + p -  q a^{2} c + 2 q a c + \\
&&+q b c - 3 q c + 2 q - 9 a^{2} c + 12 a c + 9 b c - 10 c,\\
e_3:&&-3 d p -  d q - 9 d - 3 p a^{2} b + 5 p a b + 3 p a c + 3 p b^{2} - 6 p b - 5 p c + 3 p -  q a^{2} b + 2 q a b + q a c + \\
&&+q b^{2} - 3 q b - 2 q c + 3 q - 9 a^{2} b + 12 a b + 9 a c + 9 b^{2} - 10 b - 12 c + 1,\\
e_4:&& -3 p a^{3} + 5 p a^{2} + 6 p a b - 6 p a - 5 p b - 3 p c + 6 p -  q a^{3} + 2 q a^{2} + 2 q a b - 3 q a - 2 q b -  q c + 4 q - \\
&&-9 a^{3} + 12 a^{2} + 18 a b - 10 a - 12 b - 9 c + 4.
\end{eqnarray*}}
The Gr\"obner basis for $<e_1,e_2,e_3,e_4>$ contains 19 polynomials, one of these factors as follows
\begin{eqnarray*}
&&\left(\frac{1}{233}\right) \cdot (a - 2 b + c) \cdot \\
&&\cdot (233 a^{4} - 352 a^{3} b + 108 a^{3} c + 168 a^{3} + 368 a^{2} b^{2} - 264 a^{2} b c -\\
&&- 624 a^{2} b + 46 a^{2} c^{2} - 184 a^{2} c - 544 a^{2} - 160 a b^{3} + 128 a b^{2} c + \\
&&352 a b^{2} - 16 a b c^{2} + 64 a b c + 128 a b - 4 a c^{3} - 8 a c^{2} + 768 a c +\\ &&+640 a + 48 b^{4} - 64 b^{3} c - 256 b^{3} + 32 b^{2} c^{2} + 288 b^{2} c + 384 b^{2} -\\
&&-8 b c^{3} - 144 b c^{2} - 512 b c + c^{4} + 24 c^{3} + 96 c^{2} - 640 c - 256).
\end{eqnarray*}
Let us consider the case $c=2b-a.$
Denote by $e_{1,c},e_{2,c},e_{3,c},e_{4,c}$ the polynomials obtained by substituting $c=2b-a$ into $e_1,e_2,e_3$ and $e_4.$ 
Let us denote by $G_{c}$ the Gr\"obner basis for $<e_{1,c},e_{2,c},e_{3,c},e_{4,c}>$ and compute the ideal $I_{c,p,q}=G_{c}\cap \Q[a,b,d]$, i.e., we eliminate the variables $p, q$. We get that
$$
I_{c,p,q}=<(9b-12a-3d+5)^2 - 4(3a-2)^2+48d>.
$$ 
The equation $(9b-12a-3d+5)^2 -4(3a-2)^2+48d=0$ defines the curve, say $C$, defined over $\Q(a)$ of genus 0 (in the plane $(b,d)$). The standard method allows us to find the parametrization of $C$ in the following form
$$
b=\frac{1}{36}(9a^2+36a-16-8u-u^2),\quad d=\frac{1}{36}(9a^2+36a-16+8u-u^2).
$$
However, with $b, d$ given above and the corresponding $c=2b-a$ we get 
$$
f(x)=\frac{1}{36}(6x+u+3a-2)\left(6x^3+(3a-u+2)x^2+2(3a-u-1)x+3(3a-u-2)\right),
$$
a reducible polynomial.

Let us consider the second factor that is
\begin{eqnarray}\label{defF}
	F(a,b,c)&=&233 a^{4} - 352 a^{3} b + 108 a^{3} c + 168 a^{3} + 368 a^{2} b^{2} - 264 a^{2} b c -\\\notag
	&&- 624 a^{2} b + 46 a^{2} c^{2} - 184 a^{2} c - 544 a^{2} - 160 a b^{3} + 128 a b^{2} c + \\\notag
	&&352 a b^{2} - 16 a b c^{2} + 64 a b c + 128 a b - 4 a c^{3} - 8 a c^{2} + 768 a c +\\ \notag
    &&+640 a + 48 b^{4} - 64 b^{3} c - 256 b^{3} + 32 b^{2} c^{2} + 288 b^{2} c + 384 b^{2} -\\\notag
	&&-8 b c^{3} - 144 b c^{2} - 512 b c + c^{4} + 24 c^{3} + 96 c^{2} - 640 c - 256.\notag
\end{eqnarray}
First we compute the polynomial for some small fixed values of $a$. It turns out that $F(2,b,c)$ is a reducible polynomial given by
$$
(12b^2-4bc-96b+c^2+12c+196)(4b^2-4bc-16b+c^2+4c+20).
$$
Let us study this special case when $a=2$. Consider the equation $12b^2-4bc-96b+c^2+12c+196=0$.
It follows that $(c-2b+6)^2+2(2b-9)^2=2$. The only integral solutions correspond with $b=4$ or $b=5$. If $b=4$, then $c=2$ and $d=3$. We obtain the reducible polynomial $x^4+2x^2+4x^2+2x+3=(x^2+1)(x^2+2x+3)$. If $b=5$, then $c=4$ and $d=2$, the example from Peth\H{o}'s paper. 
The set of rational solutions of $(c-2b+6)^2+2(2b-9)^2=2$ can be easily parametrized with
$$
b=\frac{8 t^2+5}{2 t^2+1},\quad c=\frac{4(t^2-t+1)}{2 t^2+1}.
$$ 
With $a=2$ and $b, c$ given above we easily compute the values 
\begin{eqnarray*}
	d&=&\frac{2 \left(3 t^2-2 t+1\right)}{2 t^2+1},\\
	p&=&\frac{4 \left(2 t^3-5 t^2+t-1\right)}{4 t^2+1},\\
	q&=&-\frac{2 (4 t-1) \left(3
   t^2-2 t+1\right)}{4 t^2+1}.
\end{eqnarray*}
With $p, q$ given above one can easily check that the discriminant of $x^2+px+q$ is positive for all $t\in\R$ (and thus for all $t\in\Q$).

Consider the other possibility, that is the equation $4b^2-4bc-16b+c^2+4c+20=0.$ We have
$$
(2b-c)^2+20=4(4b-c).
$$
Let $u=2b-c$ and $v=4b-c.$ We get that $v=\frac{u^2+20}{4}$
and $b=\frac{u^2-4u+20}{8}, c=\frac{u^2-8u+20}{4}.$ Thus
\begin{eqnarray*}
	b&=&2t^2+2,\\
	c&=&4t^2-4t+2,
\end{eqnarray*}
where $u=4t+2.$ Let us denote by $e_1',e_2',e_3'$ and $e_4'$ the corresponding polynomials $e_1,e_2,e_3$ and $e_4$ after the substitution $a=2, b=2t^2+2, c=4t^2-4t+2$. Let $G'$ be the Gr\"obner basis of the ideal $<e_1', e_2', e_3', e_4'>$ with respect to the variables $d, p, q$ over polynomial ring $\Q[t]$. We get that 
$$
G'\cap \Q[t][d]=<-t(1-d-4t+6t^2)(t^3+t^2-t-2), (d-6t^2+4t-1)(7+d+12t-6t^2-8t^3)>
$$
and thus $d=6t^2 - 4t + 1$ or $t=0$.
 
If $t=0,$ then $d=7$ and $f(x)$ is reducible $x^4+2x^3+2x^2+2x-7=(x-1)(x^3+3x^2+5x+7)$, a contradiction. If $d=6t^2 - 4t + 1$, then we have an infinite family of solutions of Peth\H{o}'s question given by
\begin{eqnarray*}
	a&=&2,\\
	b&=&2t^2+2,\\
	c&=&4t^2-4t+2,\\
	d&=&6t^2-4t+1,\\
	p&=&-\frac{6 \, t^{2} - 6 \, t + 1}{t^{2} - t},\\
	q&=&\frac{18 \, t^{3} - 18 \, t^{2} + 7 \, t - 1}{2 \, {\left(t^{3} - t^{2}\right)}}.
\end{eqnarray*}
It follows from Lemma \ref{f1} that there are infinitely many irreducible polynomial in this family. By computing the discriminant of the polynomial $x^2+px+q$ we observe that it has two real roots for $t\in\Q$ satisfying $t\in (1-\sqrt{2}/2, 1+\sqrt{2}/2) \setminus \{1\}$.

We computed all integral solutions of the equation $F(a,b,c)$ with $-200\leq a,b\leq 200.$ If $a=2,$ then we have all solutions provided by the above formulas and we also obtain $a=b=c=2$ and $d=-7.$ The corresponding polynomial is reducible, it is $(x-1)(x^3 + 3x^2 + 5x + 7).$ The remaining solutions are contained in Table 1.
\begin{equation*}
\begin{array}{|l|l|l|}
\hline
 (-30, 197, 420, 706)   &            (-12, 26, 60, 121) &    (6, 17, 24, 22) \\

(-28, 170, 364, 617)   &            (-10, 17, 40, 86)  &    (8, 26, 40, 41)\\

(-26, 145, 312, 534)   &            (-8, 10, 24, 57)   &    (10, 37, 60, 66)\\

(-24, 122, 264, 457)   &           (-6, 5, 12, 34)    &    (12, 50, 84, 97)\\

(-22, 101, 220, 386)   &            (-4, 2, 4, 17)     &    (14, 65, 112, 134)\\

(-20, 82, 180, 321)    &            (-2, 1, 0, 6)      &   (16, 82, 144, 177)\\

(-18, 65, 144, 262)    &            (0, 2, 0, 1)       &   (18, 101, 180, 226) \\

(-16, 50, 112, 209)    &            (2, 5, 4, 2)       &  \\

(-14, 37, 84, 162)     &            (4, 10, 12, 9)     &\\
\hline
\end{array}
\end{equation*}
\begin{center}
Table 1. Integral solutions of the equation $F(a,b,c)$ with $-200\leq a,b\leq 200$.
\end{center}

All these solutions can be described by the formulas
\begin{eqnarray}\label{secondpar}
	a&=&2t,\\\notag
	b&=&t^2+2t+2,\\\notag
	c&=&2t^2+2t,\\\notag
	d&=&3t^2-2t+1,\\\notag
	p&=&-\frac{2 \, {\left(3 \, t^{2} - 5 \, t + 4\right)}}{t^{2} - 2 \, t + 2},\\\notag
	q&=&\frac{9 \, t^{3} - 12 \, t^{2} + 7 \, t - 2}{t^{3} - 2 \, t^{2} + 2 \, t}.\notag
\end{eqnarray}
It follows from Lemma \ref{f2} that there are infinitely many irreducible polynomial in this family.
By computing the discriminant of the polynomial $x^2+px+q$ we observe that it has two real roots for $t\in\Q$ satisfying $t\in (0,2)$.

\end{proof}

\begin{rem}
{\rm We extended the search of the solutions of $F(a,b,c)=0$ up to $-10^4\leq a,b\leq 10^4$ and found no additional solutions. }
\end{rem}

\begin{rem}
One can prove that the polynomial $f(x)=x^4+ax^3+bx^2+cx+d$ with
$$
a=2,\quad  b=\frac{8 t^2+5}{2 t^2+1},\quad c=\frac{4(t^2-t+1)}{2 t^2+1}, \quad d=\frac{2 \left(3 t^2-2 t+1\right)}{2 t^2+1}.
$$
has no rational roots for all $t\in\Q$. However, if $t=(2 - s^2)/(4s)$, where $s\in\Q\setminus\{0\}$, then
$$
f(x)=\left(x^2+\frac{4}{s^2+2}x+\frac{s^2+4 s+6}{s^2+2}\right)\left(x^2+\frac{2 s^2}{s^2+2}x+\frac{3 s^2-4 s+2}{s^2+2}\right).
$$
\end{rem}

Let $\mathbb{P}$ be the set of prime numbers, $S\subset \mathbb{P}\cup\{\infty\}$, and recall that a rational number $r=r_1/r_2\in\Q, \gcd(r_1,r_2)=1$, is called $S$-integral if the set of prime factors of $r_2$ is a subset of $S$. The set of $S$-integers is denoted by $\Z_{S}$.

Although we were unable to prove that there are infinitely many quartic algebraic integers $\alpha$ such that the number $\beta=4\alpha^4/(1-\alpha^4)-\alpha/(\alpha-1)$  is real quadratic, from our result we can deduce the following:

\begin{cor}
Let $S\subset \mathbb{P}$. Then there are infinitely many $a,b,c,d\in\Z_{S}$ such that for one of the roots of $x^4+ax^3+bx^2+cx+d=0$, say $\alpha$, the number $\beta$ is real quadratic.
\end{cor}
\begin{proof}
In order to get the result it is enough to use the parametrization (\ref{secondpar}) by taking $t\in\Z_{S}$ satisfying the condition $t\in (0,2)$. Because there are infinitely many such $t$'s we get the result.

\end{proof}


\begin{rem}
{\rm
Let us note that the equation $F(a,b,c)=0$, where $F$ is given by (\ref{defF}),  defines (an affine) quartic surface, say $V$. The existence of the parametric solution presented above leads to the generic point (by taking $t=a/2$):
$$
(a,b,c)=\left(a,\frac{a^2}{4}+a+2,\frac{a^2}{2}+a\right)
$$
lying on $V$. This suggests to look on $V$ as on a {\it quartic curve} defined over the rational function field $\Q(a)$. We call this curve $\cal{C}$. A quick computation in MAGMA \cite{MAGMA} reveals that the genus of $\cal{C}$ is 0. This implies that $\cal{C}$ is $\overline{\Q(a)}$-rational curve. Moreover, the existence of $\Q(a)$-rational point on $\cal{C}$ given by $P=\left(\frac{a^2}{4}+a+2,\frac{a^2}{2}+a\right)$ allows us to compute rational parametrization which is defined over $\Q(a)$ as follows

\begin{eqnarray*}
b(t)&=&\frac{\sum_{i=0}^{6}bn_i(t)a^i}{\sum_{i=0}^{4}bd_i(t)a^i},\\ c(t)&=&\frac{\sum_{i=0}^{6}cn_i(t)a^i}{\sum_{i=0}^{4}cd_i(t)a^i},\\
d(t)&=&\frac{\sum_{i=0}^{6}dn_i(t)a^i}{\sum_{i=0}^{4}dd_i(t)a^i}.
\end{eqnarray*}
where $t\in\mathbb{Q}$ and $bn_i(t),bd_i(t)$ are given by
\begin{center}
\tiny
\begin{tabular}{|l|l|l|}
	\hline
	$i$ & $bn_i(t)$ & $bd_i(t)$\\
	\hline
	\hline
	$0$ & $663552 \, t^{4} - 2211840 \, t^{3} + 2764800 \, t^{2} - 1536000 \, t + 320000$ & $331776 \, t^{4} - 1105920 \, t^{3} + 1382400 \, t^{2} - 768000 \, t + 160000$ \\
	$1$ & $-331776 \, t^{4} + 1050624 \, t^{3} - 1244160 \, t^{2} + 652800 \, t - 128000$ & $-331776 \, t^{4} + 1050624 \, t^{3} - 1244160 \, t^{2} + 652800 \, t - 128000$ \\
	$2$ & $-41472 \, t^{3} + 105984 \, t^{2} - 90240 \, t + 25600$ & $124416 \, t^{4} - 373248 \, t^{3} + 419328 \, t^{2} - 209280 \, t + 39200$ \\
	$3$ & $38016 \, t^{3} - 89280 \, t^{2} + 69696 \, t - 18080$ & $-20736 \, t^{4} + 58752 \, t^{3} - 62784 \, t^{2} + 30048 \, t - 5440$ \\
	$4$ & $12960 \, t^{4} - 47520 \, t^{3} + 62928 \, t^{2} - 36240 \, t + 7748$ & $1296 \, t^{4} - 3456 \, t^{3} + 3528 \, t^{2} - 1632 \, t + 288$ \\
	$5$ & $-3888 \, t^{4} + 11664 \, t^{3} - 13248 \, t^{2} + 6792 \, t - 1332$ & $0$ \\
	$6$ & $324 \, t^{4} - 864 \, t^{3} + 900 \, t^{2} - 432 \, t + 81$ & $0$ \\
	\hline
\end{tabular}	
\end{center}
$cn_i(t)$ and $cd_i(t)$ are as follows
\begin{center}
	\tiny
\begin{tabular}{|l|l|l|}
	\hline
	$i$ & $cn_i(t)$ & $cd_i(t)$\\
	\hline
	\hline
	$0$ & $0$ & $165888 \, t^{4} - 552960 \, t^{3} + 691200 \, t^{2} - 384000 \, t + 80000$ \\
	$1$ & $165888 \, t^{4} - 552960 \, t^{3} + 691200 \, t^{2} - 384000 \, t + 80000$ & $-165888 \, t^{4} + 525312 \, t^{3} - 622080 \, t^{2} + 326400 \, t - 64000$ \\
	$2$ & $-82944 \, t^{4} + 235008 \, t^{3} - 241920 \, t^{2} + 105600 \, t - 16000$ & $62208 \, t^{4} - 186624 \, t^{3} + 209664 \, t^{2} - 104640 \, t + 19600$ \\
	$3$ & $-20736 \, t^{4} + 86400 \, t^{3} - 126720 \, t^{2} + 79296 \, t - 18080$ & $-10368 \, t^{4} + 29376 \, t^{3} - 31392 \, t^{2} + 15024 \, t - 2720$ \\
	$4$ & $20736 \, t^{4} - 66528 \, t^{3} + 79920 \, t^{2} - 42744 \, t + 8620$ & $648 \, t^{4} - 1728 \, t^{3} + 1764 \, t^{2} - 816 \, t + 144$ \\
	$5$ & $-4536 \, t^{4} + 13176 \, t^{3} - 14580 \, t^{2} + 7308 \, t - 1404$ & $0$ \\
	$6$ & $324 \, t^{4} - 864 \, t^{3} + 900 \, t^{2} - 432 \, t + 81$ & $0$ \\
	\hline
\end{tabular}
\end{center}
$dn_i(t)$ and $dd_i(t)$ are given by
\begin{center}
	\tiny
	\begin{tabular}{|l|l|l|}
		\hline
		$i$ & $cn_i(t)$ & $cd_i(t)$\\
		\hline
		\hline
	$0$ & $331776 \, t^{4} - 1105920 \, t^{3} + 1382400 \, t^{2} - 768000 \, t + 160000$ & $331776 \, t^{4} - 1105920 \, t^{3} + 1382400 \, t^{2} - 768000 \, t + 160000$ \\
	$1$ & $-663552 \, t^{4} + 2211840 \, t^{3} - 2764800 \, t^{2} + 1536000 \, t - 320000$ & $-331776 \, t^{4} + 1050624 \, t^{3} - 1244160 \, t^{2} + 652800 \, t - 128000$ \\
	$2$ & $705024 \, t^{4} - 2350080 \, t^{3} + 2939904 \, t^{2} - 1635840 \, t + 341600$ & $124416 \, t^{4} - 373248 \, t^{3} + 419328 \, t^{2} - 209280 \, t + 39200$ \\
	$3$ & $-393984 \, t^{4} + 1271808 \, t^{3} - 1540224 \, t^{2} + 829536 \, t - 167680$ & $-20736 \, t^{4} + 58752 \, t^{3} - 62784 \, t^{2} + 30048 \, t - 5440$ \\
	$4$ & $115344 \, t^{4} - 353376 \, t^{3} + 407880 \, t^{2} - 210624 \, t + 41148$ & $1296 \, t^{4} - 3456 \, t^{3} + 3528 \, t^{2} - 1632 \, t + 288$ \\
	$5$ & $-16848 \, t^{4} + 48384 \, t^{3} - 52992 \, t^{2} + 26304 \, t - 5004$ & $0$ \\
	$6$ & $972 \, t^{4} - 2592 \, t^{3} + 2700 \, t^{2} - 1296 \, t + 243$ & $0$ \\
	\hline
\end{tabular}
\end{center}
The above parametrizations yield formulas for $p$ and $q$ as well, we have
\begin{eqnarray*}
	p(t)&=&\frac{\sum_{i=0}^{8}pn_i(t)a^i}{\sum_{i=0}^{8}pd_i(t)a^i},\\ q(t)&=&\frac{\sum_{i=0}^{8}qn_i(t)a^i}{\sum_{i=0}^{8}qd_i(t)a^i},
\end{eqnarray*}
where $pn_i(t)$ and $pd_i(t)$ are as follows
\begin{center}
	\tiny
	\begin{tabular}{|l|l|}
		\hline
		$i$ & $pn_i(t)$\\
		\hline
		\hline

			$0$ & $-764411904 \, t^{6} + 3853910016 \, t^{5} - 8095334400 \, t^{4} + 9068544000 \, t^{3} - 5713920000 \, t^{2} + 1920000000 \, t - 268800000$ \\
			$1$ & $1624375296 \, t^{6} - 8066138112 \, t^{5} + 16689659904 \, t^{4} - 18417991680 \, t^{3} + 11433369600 \, t^{2} - 3785472000 \, t + 522240000$ \\
			$2$ & $-1576599552 \, t^{6} + 7711801344 \, t^{5} - 15724855296 \, t^{4} + 17109688320 \, t^{3} - 10477670400 \, t^{2} + 3424128000 \, t - 466560000$ \\
			$3$ & $901767168 \, t^{6} - 4328681472 \, t^{5} + 8665989120 \, t^{4} - 9262688256 \, t^{3} + 5575491072 \, t^{2} - 1792177920 \, t + 240364800$ \\
			$4$ & $-328458240 \, t^{6} + 1538403840 \, t^{5} - 3007901952 \, t^{4} + 3143418624 \, t^{3} - 1852477056 \, t^{2} + 583908864 \, t - 76936320$ \\
			$5$ & $77262336 \, t^{6} - 350884224 \, t^{5} + 666600192 \, t^{4} - 678507840 \, t^{3} + 390510720 \, t^{2} - 120573504 \, t + 15612432$ \\
			$6$ & $-11384064 \, t^{6} + 49828608 \, t^{5} - 91598688 \, t^{4} + 90593856 \, t^{3} - 50882400 \, t^{2} + 15399264 \, t - 1963512$ \\
			$7$ & $956448 \, t^{6} - 4012416 \, t^{5} + 7116336 \, t^{4} - 6832512 \, t^{3} + 3747096 \, t^{2} - 1113696 \, t + 140292$ \\
			$8$ & $-34992 \, t^{6} + 139968 \, t^{5} - 239112 \, t^{4} + 222912 \, t^{3} - 119556 \, t^{2} + 34992 \, t - 4374$ \\
	
	\hline	
\end{tabular}
\end{center}
\begin{center}
	\tiny
	\begin{tabular}{|l|l|}
		\hline
		$i$ & $pd_i(t)$\\
		\hline
		\hline
	$0$ & $191102976 \, t^{6} - 955514880 \, t^{5} + 1990656000 \, t^{4} - 2211840000 \, t^{3} + 1382400000 \, t^{2} - 460800000 \, t + 64000000$ \\
	$1$ & $-382205952 \, t^{6} + 1879179264 \, t^{5} - 3849928704 \, t^{4} + 4206919680 \, t^{3} - 2586009600 \, t^{2} + 847872000 \, t - 115840000$ \\
	$2$ & $346374144 \, t^{6} - 1676132352 \, t^{5} + 3381460992 \, t^{4} - 3640578048 \, t^{3} + 2206264320 \, t^{2} - 713625600 \, t + 96256000$ \\
	$3$ & $-185131008 \, t^{6} + 879869952 \, t^{5} - 1744478208 \, t^{4} + 1847079936 \, t^{3} - 1101689856 \, t^{2} + 351010560 \, t - 46678400$ \\
	$4$ & $63452160 \, t^{6} - 294865920 \, t^{5} + 572209920 \, t^{4} - 593720064 \, t^{3} + 347511168 \, t^{2} - 108828288 \, t + 14251040$ \\
	$5$ & $-14183424 \, t^{6} + 64074240 \, t^{5} - 121124160 \, t^{4} + 122713920 \, t^{3} - 70316928 \, t^{2} + 21620544 \, t - 2788424$ \\
	$6$ & $2006208 \, t^{6} - 8755776 \, t^{5} + 16052256 \, t^{4} - 15836256 \, t^{3} + 8873352 \, t^{2} - 2679360 \, t + 340884$ \\
	$7$ & $-163296 \, t^{6} + 684288 \, t^{5} - 1212408 \, t^{4} + 1162944 \, t^{3} - 637200 \, t^{2} + 189216 \, t - 23814$ \\
	$8$ & $5832 \, t^{6} - 23328 \, t^{5} + 39852 \, t^{4} - 37152 \, t^{3} + 19926 \, t^{2} - 5832 \, t + 729$ \\
		\hline	
\end{tabular}
\end{center}
finally, the formulas for $qn_i(t)$ and $qd_i(t)$
\begin{center}
	\tiny
	\begin{tabular}{|l|l|}
		\hline
		$i$ & $qn_i(t)$\\
		\hline
		\hline
	$0$ & $-382205952 \, t^{6} + 1911029760 \, t^{5} - 3981312000 \, t^{4} + 4423680000 \, t^{3} - 2764800000 \, t^{2} + 921600000 \, t - 128000000$ \\
	$1$ & $1242169344 \, t^{6} - 6210846720 \, t^{5} + 12939264000 \, t^{4} - 14376960000 \, t^{3} + 8985600000 \, t^{2} - 2995200000 \, t + 416000000$ \\
	$2$ & $-1934917632 \, t^{6} + 9650700288 \, t^{5} - 20059840512 \, t^{4} + 22242263040 \, t^{3} - 13875148800 \, t^{2} + 4617216000 \, t - 640320000$ \\
	$3$ & $1821450240 \, t^{6} - 9005727744 \, t^{5} + 18560544768 \, t^{4} - 20410417152 \, t^{3} + 12630919680 \, t^{2} - 4170854400 \, t + 574144000$ \\
	$4$ & $-1091377152 \, t^{6} + 5298628608 \, t^{5} - 10725198336 \, t^{4} + 11586309120 \, t^{3} - 7046019072 \, t^{2} + 2287269120 \, t - 309668800$ \\
	$5$ & $422143488 \, t^{6} - 1995010560 \, t^{5} + 3934065024 \, t^{4} - 4144690944 \, t^{3} + 2461317120 \, t^{2} - 781454784 \, t + 103672320$ \\
	$6$ & $-104789376 \, t^{6} + 478690560 \, t^{5} - 914397984 \, t^{4} + 935521920 \, t^{3} - 541037520 \, t^{2} + 167812512 \, t - 21823272$ \\
	$7$ & $16119648 \, t^{6} - 70777152 \, t^{5} + 130483872 \, t^{4} - 129400416 \, t^{3} + 72863136 \, t^{2} - 22105152 \, t + 2825172$ \\
	$8$ & $-1399680 \, t^{6} + 5878656 \, t^{5} - 10437336 \, t^{4} + 10031040 \, t^{3} - 5506488 \, t^{2} + 1638144 \, t - 206550$ \\
		\hline
\end{tabular}
\end{center}

\begin{center}
	\tiny
	\begin{tabular}{|l|l|}
		\hline
		$i$ & $qd_i(t)$\\
		\hline
		\hline
$0$ & $0$ \\
$1$ & $191102976 \, t^{6} - 955514880 \, t^{5} + 1990656000 \, t^{4} - 2211840000 \, t^{3} + 1382400000 \, t^{2} - 460800000 \, t + 64000000$ \\
$2$ & $-382205952 \, t^{6} + 1879179264 \, t^{5} - 3849928704 \, t^{4} + 4206919680 \, t^{3} - 2586009600 \, t^{2} + 847872000 \, t - 115840000$ \\
$3$ & $346374144 \, t^{6} - 1676132352 \, t^{5} + 3381460992 \, t^{4} - 3640578048 \, t^{3} + 2206264320 \, t^{2} - 713625600 \, t + 96256000$ \\
$4$ & $-185131008 \, t^{6} + 879869952 \, t^{5} - 1744478208 \, t^{4} + 1847079936 \, t^{3} - 1101689856 \, t^{2} + 351010560 \, t - 46678400$ \\
$5$ & $63452160 \, t^{6} - 294865920 \, t^{5} + 572209920 \, t^{4} - 593720064 \, t^{3} + 347511168 \, t^{2} - 108828288 \, t + 14251040$ \\
$6$ & $-14183424 \, t^{6} + 64074240 \, t^{5} - 121124160 \, t^{4} + 122713920 \, t^{3} - 70316928 \, t^{2} + 21620544 \, t - 2788424$ \\
$7$ & $2006208 \, t^{6} - 8755776 \, t^{5} + 16052256 \, t^{4} - 15836256 \, t^{3} + 8873352 \, t^{2} - 2679360 \, t + 340884$ \\
$8$ & $-163296 \, t^{6} + 684288 \, t^{5} - 1212408 \, t^{4} + 1162944 \, t^{3} - 637200 \, t^{2} + 189216 \, t - 23814$ \\
\hline
\end{tabular}
\end{center}
The reader interested in the details of mathematics behind the computation of parametrizations of rational curves can consult the excellent book of Rafael Sendra, Winkler and P\'{e}rez-D\'{i}az \cite{RATCurves}.
Let us also note that for $p,q$ given above the discriminant of $P(x)=x^2+px+q$ takes the form
$$
\operatorname{Disc}(P)=-2a P_1(a,t)\cdot P_{2}(a,t) Q(a,t)^2,
$$
where $Q$ is the rational function, $P_{2}$ is the polynomial of degree 2 (with respect to the variable $t$) with negative discriminant for $a\in\R\setminus\{4\}$ and
$$
P_1(a,t)=(9a^3-116a^2+524a-800)t^2-24(a-5)(a-4)^2t+18 (a-4)^3.
$$
We thus see that the polynomial $P(x)$ will have two real roots iff $-aP_1(a,t)>0$ and $Q(a,t)\neq 0$. We observe that if $a<0$ then $-aP(a,t)$ is always negative and we get no solutions. Indeed, if $a<0$ then $P(a,t)$ need to be positive. However, $9a^3-116a^2+524a-800<0$ and $\operatorname{Disc}_{t}(P_1)=-72(a-4)^3(a-2)^2a<0$ and thus $P_1(a,t)<0$ for all $a,t\in\R$. If $a>0$ and $a\neq 4$ there are solutions but the analytic expressions are quite complicated. Instead, in Figure \ref{fig:disc}, we present a plot of the solutions of the system $-aP(a,t)>0 \wedge Q(a,t)\neq 0$ satisfying $(a,t)\in\ [0,10]\times [-10,10]$. In particular, if $a\in\ (0,2)$ and $t\in\Q$ we get solutions we are interested in. 
Unfortunately, we were no able to characterize all pairs $(a,t)$ such that the corresponding polynomial $f(x)=x^4+ax^3+b(t)x^2+c(t)x+d(t)$ is irreducible. It seems that this is rather difficult question.
 
Finally, if $a=4$ then we get $(b,c,d,p,q)=(46/3, 20, 25, 165/26, 525/52)$ and the polynomial $x^2+px+q$ has complex roots.
\begin{figure}[htbp] 
       \centering
         \includegraphics[width=3in]{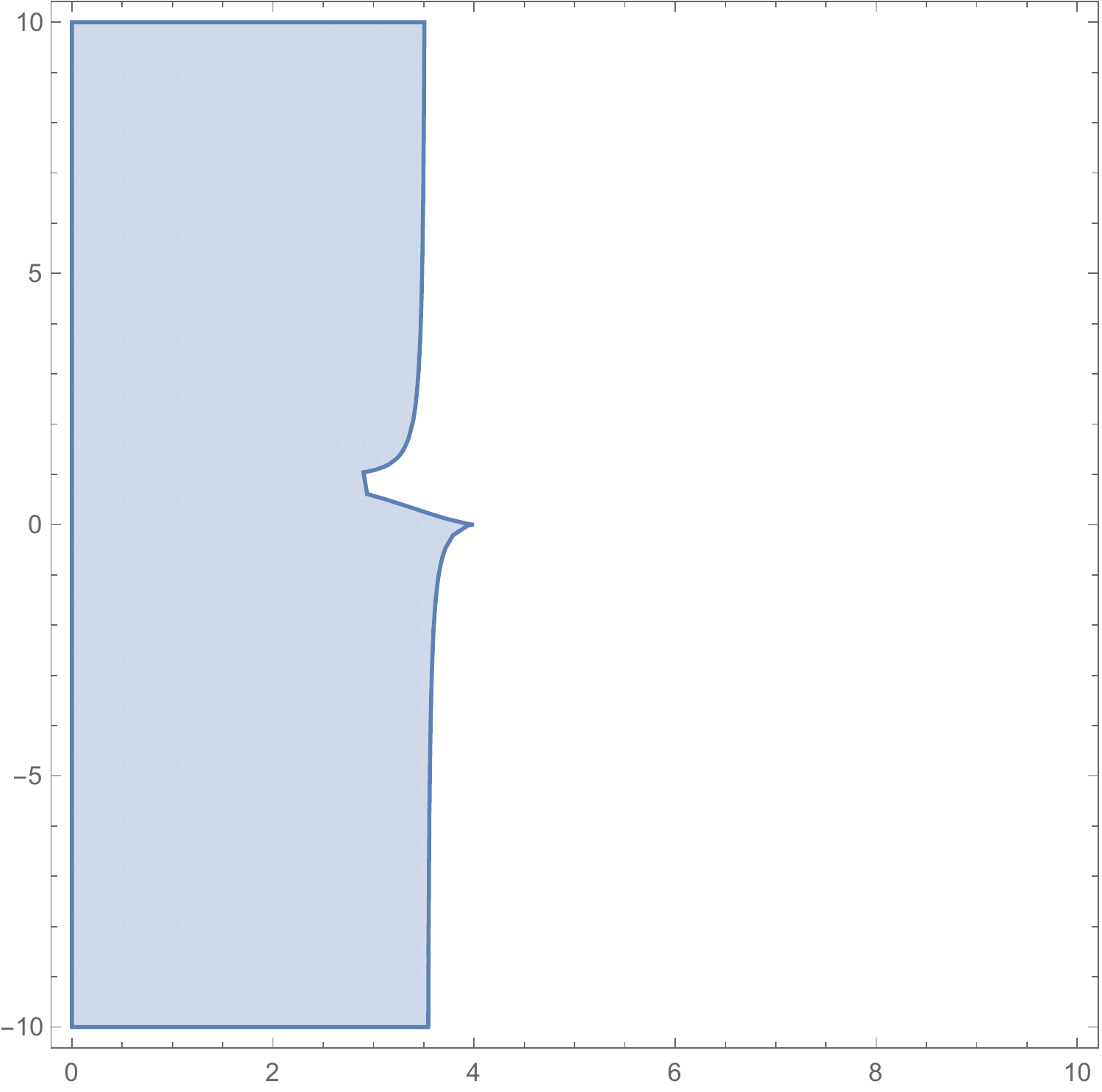}
        \caption{Real solutions of the inequality $-aP_1(a,t)>0, (a,t)\in [0,10]\times [-10,10]$, are in shaded region}
       \label{fig:disc}
    \end{figure}

We were trying to use the obtained parametrization to find other integer points on the surface $V$ but without success. If $\alpha$ is not an algebraic integer, then using the above parametrizations we may obtain real quadratic algebraic numbers. Indeed, if $\alpha$ is a root of the polynomial $x^4+ax^3+bx^2+cx+d$ then write $\beta=\frac{4\alpha^4}{\alpha^4-1}-\frac{\alpha}{\alpha-1}$. As an example let us consider the case $a=1, t=1$. The above formulas provide that $\alpha$ is a root of the polynomial
$$
x^4 + x^3 + 97/24x^2 + 3/4x + 17/8
$$
then $\beta$ is a root of the following polynomial having two real roots
$$
x^2 - 6/13x - 51/5.
$$

We can also notice "near misses" solutions of Peth\H{o}'s problem, where among the numbers $a, b, c, d$ only one is genuine rational. All these solutions correspond to $a=2$. More precisely, if $\alpha$ is solution of
$$
x^4+2x^3+\frac{14}{3}x^2+2x+1
$$
then $\beta$ is root of the polynomial
$$
x^2+3x-\frac{3}{4}.
$$
Similarly, if $\alpha$ is a root of
$$
x^4+2x^3+\frac{13}{3}x^2+4x+4
$$
then $\beta$ is a root of
$$
x^2+\frac{36}{5}x+12.
$$
}
\end{rem}
\bibliography{all}
\end{document}